\pdfoutput=1

\documentclass[12pt]{article}
\usepackage[margin=1in]{geometry}
\usepackage{amsmath}
\usepackage{amsfonts}
\usepackage{amssymb}
\usepackage{amsthm}
\usepackage{esint}
\usepackage{mathtools}
\usepackage{array,makecell,longtable}
\usepackage{titling}
\usepackage{authblk}
\usepackage{commath}
\usepackage{xcolor}
\usepackage{graphicx}
\usepackage{graphbox}
\usepackage[toc,title,page]{appendix}
\usepackage{subcaption}
\usepackage{bbm}
\usepackage{hyperref}
\usepackage{color,soul}
\usepackage{bm}
\usepackage{mathrsfs}

\usepackage{enumitem}
\newtheorem{remark}{Remark}[section]
\newtheorem{theorem}[remark]{Theorem}

\newtheorem{definition}[remark]{Definition}
\newtheorem{proposition}[remark]{Proposition}

\numberwithin{equation}{section}
\allowdisplaybreaks

\title{Determining a parabolic system by boundary observation of its non-negative solutions with biological applications}

\author{Hongyu Liu\thanks{Department of Mathematics, City University of Hong Kong, Hong Kong SAR, China\\ Email address: hongyu.liuip@gmail.com, hongyliu@cityu.edu.hk} \, and Catharine W.K. Lo\thanks{Liu Bie Ju Centre for Mathematical Sciences, City University of Hong Kong, Hong Kong SAR, China\\ Email address: wingkclo@cityu.edu.hk} \vspace{-0.5cm}}

\date{}

\begin{document}

\maketitle

\begin{abstract}
    In this paper, we consider the inverse problem of determining some coefficients within a coupled nonlinear parabolic system, through boundary observation of its non-negative solutions. In the physical setup, the non-negative solutions represent certain probability densities in different contexts. We innovate the successive linearisation method by further developing a high-order variation scheme which can both ensure the positivity of the solutions and effectively tackle the nonlinear inverse problem. This enables us to establish several novel unique identifiability results for the inverse problem in a rather general setup.  For a theoretical perspective, our study addresses an important topic in PDE analysis on how to characterise the function spaces generated by the products of non-positive solutions of parabolic PDEs. As a typical and practically interesting application, we apply our general results to inverse problems for ecological population models, where the positive solutions signify the population densities. 
\medskip
\medskip


    
\end{abstract}

\section{Introduction}

\subsection{Mathematical Setup}

To provide a general picture of our study, we consider the following coupled nonlinear system of parabolic equations:
\begin{equation}\label{GeneralProb}
    \begin{cases}
        \partial_t u(x,t) = F(x,t,u,v,\nabla u, \nabla v, \Delta u, \Delta v) &\quad \text{in }\Omega\times(0,T),\\
        \partial_t v(x,t) = G(x,t,u,v,\nabla u, \nabla v, \Delta u, \Delta v)  &\quad \text{in }\Omega\times(0,T),\\
        u,v\geq0    &\quad \text{in }\bar{\Omega}\times[0,T),
    \end{cases}
\end{equation}
\sloppy where $\Omega\subset\mathbb{R}^n$, $n\geq 2$, is a bounded Lipschitz domain, $T\in(0,\infty]$, and $F(x,t,p_1,q_1,p_2,q_2,p_3,q_3)$ and $G(x,t,p_1,q_1,p_2,q_2,p_3,q_3):\Omega\times (0,T) \times \mathbb{R}^{2n+4}\to \mathbb{R}$ are real-valued functions with respect to $p_i$ and $q_i$, $i=1,2,3$. Here, the functions $u$ and $v$ are required to be non-negative in the domain $\bar{\Omega}\times[0,T]$. In the relevant physical contexts, $u$ and $v$ signify certain probability densities and hence the non-negativity is crucial. More discussion shall be given in the next subsection about this aspect.

In this paper, we are mainly concerned with the inverse problem of determining the functions $F$ and $G$, using knowledge of $u$ (or, $v$) at the boundary $\Sigma:=\partial\Omega\times(0,T)$. As such, we introduce the measurement map $\mathcal{M}^+_{F,G}$:
\begin{equation}\label{eq:mp1}
\mathcal{M}^+_{F,G}(\left.u\right\rvert_\Sigma) = \left. \left( u(x,t),\partial_\nu u(x,t)\right) \right\rvert_\Sigma, 
\end{equation}
where $``+"$ signifies that the boundary data are associated with the non-negative solutions of the coupled parabolic system \eqref{GeneralProb}. The inverse problem mentioned above can be formulated as 
\begin{equation}\label{eq:ip1}
\mathcal{M}^+_{F, G}\longrightarrow F, G. 
\end{equation}
It is emphasised that throughout the paper we fix the boundary data to be associated with $u$, and all of our results hold equally when $u$ is replaced by $v$. 

For the inverse problem \eqref{eq:ip1}, we are mainly concerned with the theoretical unique identifiability issue, which is of primary importance for a generic inverse problem. In its general formulation, the unique identifiability asks whether one can establish the following one-to-one correspondence for two configurations $(F^j,G^j)$, $j = 1, 2$:
\begin{equation}\label{MainUniquessnessProb}
\mathcal{M}^+_{F^1,G^1} = \mathcal{M}^+_{F^2,G^2} \quad\text{ if and only if }\quad (F^1,G^1) = (F^2,G^2).
\end{equation} 
 In this paper, we aim to prove, in formal terms, the following theorem. 

\begin{theorem}\label{FormalThm}
    Let $\mathcal{M}^+_{F^j,G^j}$ be the measurement map associated to \eqref{GeneralProb} for $j=1,2$. Assume $F^j,G^j\in\mathcal{A}$ such that \eqref{CoerCond} holds, where $\mathcal{A}$ is a certain admissible class of the unknowns. 
    Suppose that
    \[
    \mathcal{M}^+_{F^1,G^1}(u|_{\partial\Omega})=\mathcal{M}^+_{F^2,G^2}(u|_{\partial\Omega})\quad \mbox{for all}\ u|_{\partial\Omega}\in\mathcal{S}, 
    \]
    where $\mathcal{S}$ is a properly chosen function space on $\Sigma$. 
    Then \[(F^1,G^1) = (F^2,G^2).\]    
\end{theorem}
We shall establish sufficient conditions such that Theorem~\ref{FormalThm} holds in a certain general setup. In particular, we shall provide general characterisations of $F, G$ and $\mathcal{A}, \mathcal{S}$, and connect them to specific applications. The major novelty that distinguishes our inverse problem study from most of the existing ones lies at the following two aspects. First, we consider inverse boundary problems for coupled nonlinear parabolic partial differential equations (PDEs) with partial data (in a certain sense to be made more precise in what follows), using the high-order successive linearisation method. Second, the solutions of the PDEs are required to be non-negative, and to that end, we innovate the successive linearisation method by further developing a high-order variation scheme. For a theoretical perspective, our study addresses an important topic in PDE analysis on how to characterise the function spaces generated by the products of non-positive solutions of parabolic PDEs.  More discussion on these aspects shall be given in Sections~\ref{sect:motivation} and \eqref{sect:1.3}.

\subsection{Motivation and Background}\label{sect:motivation}

Nonlinear parabolic PDEs are used to describe a variety of physical systems. Some examples are listed in Table \ref{table1} below, including both single parabolic equations and coupled parabolic systems; see Chapter 12 of \cite{PaoNonlinearBook} for more explanation of Examples 3--7.

The equations above are commonly derived from a statistical physics viewpoint, in that the interacting agents attempt to minimise a cost functional on the macro scale via an expectation function, such that the value functions $u$ and $v$ represent the averaged-out probability density or population density of the agents. Despite that $u$ and $v$ may have different meanings depending on the physical contexts the equations are derived from, they all have a common root, and are all intrinsically non-negative.

Our problem \eqref{MainUniquessnessProb} aims to identify the uniqueness of the external parameters in these models, such as the Hamiltonian function and running cost of the mean field games system, or the reacting functions between two chemicals or population species. In cases such as the Fokker-Planck equation where the value function satisfying a single parabolic equation is a vector, we are also able to determine each component of the source functions. Furthermore, in these examples, the value functions $u$ and $v$ are required to be positive (or non-negative), since they are the value functions of objects. This positivity requirement has not been considered in other works on inverse problems, except for another previous work \cite{LiuZhang2022-InversePbMFG} of the first author. This makes our inverse problem study more realistic from a practical point of view, compared to other studies. This positivity requirement is fulfilled via a careful choice of the boundary data in our recovery of environmental forcing functions $F$ and $G$, as we will explain in the later sections.

\begin{center}
\begin{longtable}{|m{5em}|c|m{12em}|}
  \hline
  Types of PDEs & Equation & Description\\
  \hline
  \hline
  Fokker-Planck equation & \makecell{$u_t = L^*u$,\\ $L\varphi:=tr(A\Delta\varphi)+\langle b,\nabla\varphi\rangle + c\varphi$} & $L^*$ is the adjoint operator of $L$, $A$ corresponds to addictive noise (see, for instance \cite{FokkerPlanckNoise} for more detailed discussion about noise) \\ 
  \hline
  Mean field games & \makecell{$-u_t - \Delta u + H(x,\nabla u) = F(x,t,m)$,\\ $m_t - \Delta m - \nabla\cdot(m\nabla_pH(x,\nabla u)) = 0$} & $u$ is the value function of each player, $m$ describes the population distribution, $H$ is a Hamiltonian, and $F$ is the running cost function \\ 
  \hline
  Gas-liquid interaction problems & \makecell{$u_t -D_1\Delta u = f_1(x,u,v)$,\\ $v_t -D_2\Delta v = f_2(x,u,v)$,\\$f_i(x,u,v)=-\sigma_iu^mv^n+q_i(x)$,\\$m,n\geq1, q_i(x)\geq0, i=1,2$} & $D_1,D_2$ are constant diffusion coefficients, $u,v$ represent the dissolved gas and reactant respectively, $f_i(x,u,v)$ describes the $(m,n)$-th order reaction for $m,n\geq1$  \\ 
  \hline
  Belousov-Zhabotinskii oregonator model & \makecell{$u_t -D_1\Delta u = u(a-bu-cv)$,\\ $v_t -D_2\Delta v = -c'uv$} & $D_1,D_2$ are constant diffusion coefficients, $u,v$ represent the concentrations of the metal $X$ and ion $Y$, and $a,b,c,c'>0$ are forward rate constants \\ 
  \hline
  Volterra-Lotka model & \makecell{$u_t -D_1\Delta u = u(a_1-b_1u\pm c_1v)$,\\ $v_t -D_2\Delta v = v(a_2\pm b_2u-c_2v)$} & $D_1,D_2$ are constant diffusion coefficients, $u,v$ represent two competing species, and $a_i,b_i,c_i>0$ are interaction constants \\ 
  \hline
  Epidemic Kermack-McKendrick equation & \makecell{$u_t -D_1\Delta u = -a_1u-b_1u\int_\Omega K(x,\xi)v(t,\xi)d\xi$,\\ $v_t -D_2\Delta v = -a_2u-b_2u\int_\Omega K(x,\xi)v(t,\xi)d\xi$} & $D_1,D_2$ are constant diffusion coefficients, $u,v$ represent the susceptible and infective populations respectively, $a_i,b_i$ are rate constants, and $K(x,\xi)$ is a transfer function \\ 
  \hline 
  Adiabatic neuron transport equation & \makecell{$u_t -D_1\Delta u = u(a-bv)$,\\ $v_t -D_2\Delta v = cu$} & $D_1,D_2,a,b,c>0$ \\ 
  \hline
  \caption[]{Examples of Nonlinear Parabolic Equations or Systems of Positive Value Functions\label{table1}}
\end{longtable}

\end{center}

Clearly, our analysis can also be applied to nonlinear parabolic equations or systems that do not require positivity, such as in the well-known examples in Table \ref{table2} below. Yet, we recall that positive solutions tend to possess more properties in many nonlinear parabolic equations, including existence and regularity results and maximum principles, so our study of the inverse problem for positive solutions has considerable importance too in these cases. 

\begin{center}
\begin{longtable}[h]{|m{5em}|c|m{16em}|} 
  \hline
  Types of PDEs & Equation & Description\\
  \hline
  \hline
  Burgers' equation & $\bm{u}_t + \bm{u}\bm{u}_x = \nu \Delta \bm{u} + \bm{f}(x,t)$
   & $\bm{u}$ is the vectorial fluid velocity, $\nu$ is the viscosity coefficient, and $\bm{f}(x,t)$ is an external force \\
  \hline
  Allen-Cahn equation & $\phi_t = \epsilon^2 \Delta \phi - \frac{1}{\epsilon^2} W’(\phi)$ & $\phi$ is the order parameter, $\epsilon$ denotes the interfacial thickness, and $W(\phi)$ is a double-well potential \\ 
  \hline
  Fisher-KPP equation & $u_t - D \Delta u = F(u)$ & $u$ is the population density, $D$ is the diffusion coefficient, and $F(u)$ describes the growth rate \\ 
  \hline
  Nonlinear Schr\"odinger equation & $\mathrm{i}\bm{\psi}_t = - \frac{1}{2}\Delta \bm{\psi} + \kappa |\bm{\psi}|^2\bm{\psi}$ & $\bm{\psi}$ is the complex vectorial wave field, and $\kappa$ describes the focusing/defocusing of the solution \\ 
  \hline
  Hamilton-Jacobi equation & $S_t = -H(x,\nabla S,t)$ & $H$ is the Hamiltonian of a mechanical system $S$ \\ 
  \hline
  \caption[]{Examples of Other Nonlinear Parabolic Equations \label{table2}}
\end{longtable}
\end{center}

Indeed, the external parameters and source functions are one of the key parameters in parabolic systems such as in the examples above. However, in practice these external environmental factors are often unknown or only partially known for the involved agents (such as fluid particles, interacting species or reactants), and it is also impossible to measure the value functions of the agents at all points inside a region, even if the chosen domain is bounded with enough regularity. Instead, we can only measure the value functions at the boundary of some chosen domain. This motivates us to consider the inverse problem \eqref{MainUniquessnessProb} of determining these external parameters $F$ and $G$ by measurement/knowledge of the boundary data of the parabolic system that results from optimal actions, which is formally given as Theorem \ref{FormalThm}.

\subsection{Technical developments and discussion}\label{sect:1.3}
There has been much interest in inverse problems for coupled parabolic systems, and we refer to \cite{IY1998,CGR2006,BCGPY2009,BFM2014,ACM2022} and the references therein for examples in the literature of various such works. However, most of these works aim to recover the conductivity of the system, which is related to the self-diffusion coefficients of Section \ref{sect:motivation} when the PDE is written in non-divergence form. To the best of our knowledge, our work is the first attempt to recover coefficients of nonlinear coupled terms for general parabolic systems. As such, our work has a wider range of applications compared to the formulations given in these references, and the results we present are more versatile. 

Furthermore, a major technical consideration we have in this article is the enforcement of positivity (or non-negativity) for the solutions, which is crucial in many physical applications such as population models, as we will see in Section \ref{sect:app}. Moreover, many previous works on the forward problem for parabolic PDEs require the positivity of solutions to obtain further properties of the solutions including uniqueness, regularity and maximum principles (see, for instance, \cite{Positive2,PaoNonlinearBook,Positive1}). 

Despite the importance of having positivity in solutions, we emphasise that other than \cite{LiuZhang2022-InversePbMFG} and \cite{LiuZhangMFG3}, there has not yet been any work on inverse boundary problems which take into consideration the positivity of the solutions, which has important implications when applied to physical real-world models. Moreover, unlike \cite{LiuZhang2022-InversePbMFG}, our method allows the usage of the classical successive/high-order linearisation technique around a pair of trivial solutions $(0, 0)$, yet maintaining positivity of $u$ and $v$. The main idea is to consider a specific form for the input boundary value describing higher-order variation, given by 
\begin{equation}\label{UExpand}u(x,t;\varepsilon)=\sum_{l=1}^\infty\varepsilon^l f_l\quad \text{ on }\Sigma\quad \text{ for }f_1>0.\end{equation} 
Here, $f_2(x,t)$ may possibly be positive or negative at different $x,t$, but for all small positive $\varepsilon$, the positivity of $f_1$ ensures that $u(x,t;\varepsilon)>0$ on $\Sigma$. 

In fact, by allowing the input to have high-order variations $\varepsilon^l$, $l\geq2$, with respect to an asymptotic parameter $\varepsilon$, this approach allows us to primarily concentrate on the higher-order linearised system but not the first-order one. Moreover, the input of the first-order linearisation system is forced to be nonnegative, and hence the input of the higher-order linearisation system can be arbitrary. 
This is in contrast to \cite{LiuZhang2022-InversePbMFG}, which takes into account the high-order linearisation around $(0,1)$. This linearisation around non-trivial solutions restricts the usage of the technique to certain PDEs which have the nontrivial solution $(0,1)$. 
On the other hand, the technique of linearising around a pair of trivial solutions has already been extensively used in a variety of inverse problems associated with nonlinear PDEs; see e.g. \cite{ LassasLiimatainenLinSalo2021InversePbEllipticPowerNonlinear,LinLiuLiuZhang2021-InversePbSemilinearParabolic-CGOSolnsSuccessiveLinearisation,LiuMouZhang2022-InverseProblemsMeanFieldGames,LiimatainenLin2022-SemilinearClassicalInverseProblem} and the references cited therein. As such, our technique which innovates the method of high-order linearisation around a pair of trivial solutions, yet at the same time maintaining positivity, makes our result not only physically realistic, but also highly applicable to a wide range of PDEs, including the ones in Section \ref{sect:motivation} as well as those in these references cited. To the best of our knowledge, our method is the second only possible solution available, and currently the more versatile one. We believe that these technical developments shall be useful for tackling nonlinear inverse problems related to various physical systems where positivity is intrinsic in the system.

Indeed, this technical constraint of positivity makes the corresponding inverse problem radically more challenging, since one would need to construct suitable ``probing modes", such as \eqref{UExpand}, which fulfil this constraint. If we suppose instead that positivity is not required, such as in the Burgers' equation, Allen-Cahn equation, nonlinear Schr\"odinger equation or the Hamilton-Jacobi equation of Section \eqref{sect:motivation}, our method of high-order linearisation can be easily extended as in Theorem 1.3 of \cite{LinLiuLiuZhang2021-InversePbSemilinearParabolic-CGOSolnsSuccessiveLinearisation} to determine much more information about $F$ and $G$.

At the same time, it should be noted that in our case, our measurement map only involves $u$, and no information is required for $v$. This is different from that considered in \cite{LinLiuLiuZhang2021-InversePbSemilinearParabolic-CGOSolnsSuccessiveLinearisation} or \cite{LiuZhang2022-InversePbMFG}. As such, our result can, in fact, be viewed as an inverse problem involving partial data.

Finally, we note that the mathematical study of inverse problems associated with nonlinear PDEs have received considerable interest in the literature recently; see, for example, \cite{LassasLiimatainenLinSalo2021InversePbEllipticPowerNonlinear,LassasLiimatainenLinSalo2021InversePbEllipticSemilinear,LinLiuLiuZhang2021-InversePbSemilinearParabolic-CGOSolnsSuccessiveLinearisation} and the references therein. Yet, there are several salient features of our study that are worth highlighting. First, the nonlinear parabolic systems we consider are very general, and include a diverse range of time-dependent physical models. Our method enables us to determine the time-dependent environmental factors involved in these various nonlinear parabolic systems, though with some limitations. Secondly, we innovate the classic high-order linearisation method around a trivial pair of solutions which can ensure the positivity of the solutions. We believe that the mathematical strategies  and techniques developed in this article, as well as those in \cite{LiuZhangMFG3}, offer novel perspectives on inverse boundary problems in those new and intriguing contexts and have multiple physical applications in tackling inverse problems associated with coupled nonlinear PDEs in different contexts, to produce more results of both theoretical and practical importance.

The rest of the paper is organised as follows. In Section \ref{sect:results}, we fix some notations and introduce several auxiliary results leading up to the main result of the inverse problem, which is proved in Section \ref{sect:ResultsProof}. Finally, we will discuss some physical applications in Section \ref{sect:app}, as well as possible future directions in Section \ref{sect:discuss}.

\section{Preliminaries and Statement of Main Results}\label{sect:results}

\subsection{Mathematical Setup}
We restrict our study to nonlinear system of equations of the following form:

\begin{equation}\label{MainPDE}
    \begin{cases}
        \partial_t u(x,t) - \mu\Delta u(x,t) = F(x,t,u,v) &\quad \text{in }Q,\\
        \partial_t v(x,t) - \nu\Delta v(x,t) = G(x,t,u,v)  &\quad \text{in }Q,\\
        u,v\geq0    &\quad \text{in }Q,\\
        u(x,0)=u_0(x)\geq0,\quad v(x,0) = v_0(x)\geq0  &\quad \text{in }\Omega, \\
        u = f\geq0,\quad v = g\geq0 &\quad \text{on }\Sigma
    \end{cases}
\end{equation}
where $Q:=\Omega\times(0,T)$ for a bounded Lipschitz domain $\Omega\subset\mathbb{R}^n$, $\Sigma:=\partial\Omega\times(0,T)$, $T\in(0,\infty]$. Here, $\mu,\nu>0$ are positive constants, which may represent the viscosity coefficient in the Burgers' equation, thickness of the layer between two phases in the Allen-Cahn equation, additive noise in the Fokker-Planck equation, or diffusion coefficients in population or chemical models of Section \ref{sect:motivation}. The functions $F(x,t,p,q),G(x,t,p,q):\Omega\times (0,T) \times \mathbb{R}\times \mathbb{R}\to \mathbb{R}$ are analytic with respect to $p$ and $q$, and are of the form
\begin{equation}\label{Fform}F(x,t,p,q):= \sum_{\substack{m,n\geq0\\m+n\geq3}}^\infty \alpha_{mn}(x,t) p^m q^n\end{equation} and \begin{equation}\label{Gform}G(x,t,p,q):= \sum_{\substack{m,n\geq0\\m+n\geq1}}^\infty \beta_{mn}(x,t) p^m q^n,\end{equation}
such that the following coercivity condition 
\begin{equation}\label{CoerCond}\beta_{01}(x,t)\leq0\end{equation} is satisfied. 
The Dirichlet boundary conditions given here signify that the players can freely enter and leave the domain $\Omega$ via its boundary. We note that the strategy we will develop in this paper makes full use of the intrinsic structure of the parabolic system described by \eqref{MainPDE} and \eqref{Fform}--\eqref{Gform}.

Here, we would like to remark that we shall not require that $\int_\Omega u$ or $\int_\Omega v$ is equal to $1$, though $u,v$ are treated as probability or population densities. Indeed, as in \cite{LiuZhangMFG3}, one can consider the scenario that the domain consists of a family of disjoint subdomains, say $\omega_i$, $i\in\mathbb{N}$, such that the overall density on $\cup_i \omega_i$ is 1, namely $\int_{\cup_i \omega_i} u=\int_{\cup_i \omega_i} v=1$. Though those subdomains are disjoint, the agents within each subdomain can interact with those in other subdomains, say e.g. when our domain is an arbitrarily defined area within a big forest. Hence if $\Omega$ is taken to be any one of those subdomains, i.e. $\omega_i$, it is not necessary to require that $\int_\Omega u=\int_\Omega u=1$, as long as $u,v$ are nonnegative.

As explained in the previous Section, we are mainly concerned with the inverse problem of determining the coefficients $\alpha_{mn}$ and $\beta_{mn}$, using knowledge of $u$ at the boundary of some bounded domain $\Sigma$. Translated to physical terms, this means that we assume that all agents follow the parabolic system \eqref{MainPDE} (i.e. on a macro scale, they follow laws of nature), and the observer only knows the value functions of the agents at the boundary of some chosen domain. The main goal is to recover some information regarding the environment, such as source functions or forcing functions. As such, we introduce the measurement map $\mathcal{M}^+_{F,G}$
\[\mathcal{M}^+_{F,G}(\left.u\right\rvert_\Sigma) = \left. \left( u(x,t),\partial_\nu u(x,t)\right) \right\rvert_\Sigma.\]
To be more specific, one can think of measuring/observing the space-time boundary data
of $u$, from which we can determine the interacting functions $F$ and $G$ over the space-time domain $Q$.

In particular, we are mainly concerned with the unique identifiability issue, which asks whether one can establish the following one-to-one correspondence:
\[\mathcal{M}^+_{F^1,G^1} = \mathcal{M}^+_{F^2,G^2} \quad\text{ if and only if }\quad (F^1,G^1) = (F^2,G^2)\] two configurations $(F^j,G^j)$, $j = 1, 2$.

We first give the result for the corresponding forward problem. For $k\in\mathbb{N}$ and $0<\alpha<1$, recall that the H\"older space $C^{k+\alpha}(\overline{\Omega_1})$ is defined as the subspace of $C^{k}(\overline{\Omega_1})$ such that $\phi\in C^{k+\alpha}(\overline{\Omega_1})$ if and only if $D^l\phi$ exist and are H\"older continuous with exponent $\alpha$ for all $l=(l_1,l_2,\dots,l_n)\in\mathbb{N}^n$ with $|l|\leq k$, where $D^l:=\partial^{l_1}_{x_1}\partial^{l_2}_{x_2}\cdots\partial^{l_n}_{x_n}$ for $x=(x_1,\cdots,x_n)$. The corresponding norm is defined by
\[\norm{\phi}_{C^{k+\alpha}(\overline{\Omega_1})}:=\sum_{|l|\leq k}\norm{D^l\phi}_\infty+\sum_{|l|= k}\sup_{x\neq y}\frac{|D^l\phi(x)-D^l\phi(y)|}{|x-y|^\alpha}\] for the $L^\infty$ sup norm $\norm{\cdot}_\infty$. If $\phi$ depends on both the time and space variables, we write $\phi\in C^{k+\alpha,l+\beta}(Q_1)$ if $\phi$ is $C^{k+\alpha}$ in space and $C^{l+\beta}$ in time, endowed with the natural norm.

We first observe that $(0,0)$ is a solution to the problem \eqref{MainPDE}. Furthermore, we recall the following classical well-posedness result of the forward nonlinear parabolic problem, which gives the infinite differentiability of the system with respect to small variations of the given boundary data input $f$. This can be found in \cite{LadyzhenskayaSolonnikovUraltsevaBook} as Theorem VII.7.1:

\begin{theorem}\label{MainThmForwardPb}
    Suppose that the first derivatives of $F, G$ are continuous with respect to $x,t,u,v$. For $\alpha\in(0,1)$, assume $u_0,v_0\in C^{2+\alpha}(\bar{\Omega})$, $f,g\in C^{2+\alpha,1+\alpha/2}(\bar{\Sigma})$ such that $u_0,v_0,f,g\geq0$ with the compatibility conditions
    \[u_0(x)=f(x,0)\text{ and } f_t(x,0)=\mu\Delta u_0(x) + F(x,0,u_0(x),v_0(x))\quad\text{ on }\Sigma\] and
    \[v_0(x)=g(x,0)\text{ and } g_t(x,0)=\nu\Delta v_0(x) + G(x,0,u_0(x),v_0(x))\quad\text{ on }\Sigma.\]
    Then, the system \eqref{MainPDE} admits a unique non-negative solution \[(u,v)\in [C^{2+\alpha,1+\alpha/2}(\bar{Q})]^2.\] 
\end{theorem}

\subsection{Main Results}

Suppose $F$ and $G$ are analytic.
\begin{definition}
    We say that $U(x,t,p,q):\mathbb{R}^n\times\mathbb{R}\times\mathbb{C}\times\mathbb{C}\to\mathbb{C}$ is admissible, denoted by $U\in\mathcal{A}$, if:
    \begin{enumerate}[label=(\alph*)]
        \item The map $z\mapsto U(\cdot,\cdot,p,q)$ is holomorphic with value in $C^{2+\alpha,1+\alpha/2}(\bar{Q})$ for some $\alpha\in(0,1)$,
        \item $U(x,t,0,0)=0$ for all $(x,t)\in Q$.
    \end{enumerate}
    It is clear that if $U$ satisfies these two conditions, $U$ can be expanded into a power series
    \[U(x,z)=\sum_{m,n=1}^\infty U^{(m,n)}(x)\frac{p^m q^n}{(m+n)!},\]
    where $U^{(m,n)}(x,t)=\frac{\partial^m }{\partial p^m}\frac{\partial^n}{\partial q^n}U(x,t,0)\in C^{2+\alpha,1+\alpha/2}(\bar{Q})$.
\end{definition}

This admissibility condition is imposed a priori on $F$ and $G$, by extending these functions of real variables to the complex plane with respect to the $z$-variable, given by $\tilde{F}(\cdot,\cdot,p,q)$ and $\tilde{G}(\cdot,\cdot,p,q)$ respectively, and assume that they are holomorphic as functions of the complex variables $p,q$. Then, $F$ and $G$ are simply the restrictions of $\tilde{F}$ and $\tilde{G}$ to the real line respectively. Furthermore, since the image of $F$ and $G$ are in $\mathbb{R}$, we can assume that the series expansions of $\tilde{F}$ and $\tilde{G}$ are real-valued.

We are now in a position to present our main results:

\begin{theorem}\label{MainThm}
    Let $\mathcal{M}^+_{F^j,G^j}$ be the measurement map associated to \eqref{MainPDE} for $j=1,2$. Assume $F^j,G^j\in\mathcal{A}$ such that \eqref{CoerCond} holds.
    Suppose, for any 
    \[u(x,t)=\sum_{l=1}^\infty\varepsilon^l f_l\quad \text{ on }\Sigma\] 
    where $f_l\in C^{2+\alpha,1+\alpha/2}(\bar{\Sigma})$ with $|\varepsilon|$ small enough such that $f_1(x,0)=u_0(x)$ and $f_l(x,0)=0$ for $l\geq2$,
    one has
    \[\mathcal{M}^+_{F^1,G^1}(u|_{\partial\Omega})=\mathcal{M}^+_{F^2,G^2}(u|_{\partial\Omega})\quad \mbox{for all}\ u|_{\partial\Omega}\in\mathcal{S}:=C^{2+\alpha,1+\alpha/2}(\bar{\Sigma}).\]
    Fix $m=1,\dots M$, $M<\infty$.
    \begin{enumerate}
        \item For $m=1$, if $\beta_{11}=\beta_{02}\equiv0$, and $\beta_{01},\beta_{20}$ are known and fixed, then 
        \begin{equation}\beta_{10}^1(x,t)=\beta_{10}^2(x,t)\quad\text{ in }Q.\end{equation}
        \item For $m\geq2$, if 
        \begin{equation}\label{ThmAssump2}\alpha_{m_1n_1}\equiv0\quad\text{ for all }2\leq m_1+n_1\leq m, m_1\neq m,\end{equation} 
        and $\alpha_{m_2n_2}$ are known and fixed for all $m_2+n_2 = m+1$, as well as $\beta_{10}$, then it holds that 
        \begin{equation}\alpha_{m0}^1(x,t)=\alpha_{m0}^2(x,t)\quad\text{ in }Q.\end{equation}
        \item For $m\geq2$, suppose $\alpha_{mn}$ are known and fixed for all $m,n$, 
        \begin{equation}\label{ThmAssump3}\beta_{m_1n_1}\equiv0\quad\text{ for all }2\leq m_1+n_1\leq m, m_1\neq m,\end{equation} 
        and $\beta_{m_2,n_2}$ are known and fixed for all $m_2+n_2 = m+1$ or $m_2+n_2 \leq 1$. If, in addition, 
        \begin{equation}\label{ThmAssump}\text{either }\quad \alpha_{m_1n_1}\equiv0\text{ for all }2\leq m_1+n_1\leq m\quad\text{ or }\quad\beta_{10}\equiv0,\end{equation} 
        then it holds that 
        \begin{equation}\beta_{m0}^1(x,t)=\beta_{m0}^2(x,t)\quad\text{ in }Q.\end{equation}
    \end{enumerate}
    
\end{theorem}

Observe that the recovery of these coefficients is not simultaneous. On the other hand, as long as the assumptions \eqref{ThmAssump2}, \eqref{ThmAssump3} and \eqref{ThmAssump} are satisfied for some $m\geq2$, it is possible to obtain that the results of (2) and (3) separately by choosing the same $u(x,t)=\sum_{l=1}^{m+1}\varepsilon^l f_l$ on $\Sigma$.

\section{Proof of Main Results}\label{sect:ResultsProof}
To prove the main results, we will develop a high-order linearisation scheme of the system \eqref{MainPDE} with respect to $u(x,t)$ in the case that $F(x,t,u,v),G(x,t,u,v)\in \mathcal{A}$. By Theorem \eqref{MainThmForwardPb},
given $f(x,t)$, the PDE system \eqref{MainPDE} is infinitely differentiable with respect to the input  $f(x,t)$.

We introduce the basic setting of this higher order linearisation method. Consider the system \eqref{MainPDE}. Let 
\[u(x,t;\varepsilon)=\sum_{l=1}^\infty\varepsilon^l f_l\quad \text{ on }\Sigma,\]
where $f_l\in C^{2+\alpha,1+\alpha/2}(\bar{\Sigma})$ with $|\varepsilon|$ small enough satisfying \begin{equation}\label{Cond1}f_1(x,0)=u_0(x)\quad\text{ and }\quad f_l(x,0)=0\text{ for }l\geq2.\end{equation} Assume 
\begin{equation}\label{Cond2}
    f_1(x,t)>0\quad\forall x\in\Omega,t\in(0,T),
\end{equation}
so that for all small positive $\varepsilon$, $u(x,t;\varepsilon)>0$ on $\Sigma$. Then, by Theorem \ref{MainThmForwardPb}, there exists a unique solution $(u(x,t;\varepsilon), v(x,t;\varepsilon))$ of \eqref{MainPDE}. Let $(u(x,t;0), v(x,t;0))$ be the solution of \eqref{MainPDE} when $\varepsilon = 0$.

Let 
\[u^{(1)} := \left. \partial_{\varepsilon} u \right\rvert_{\varepsilon = 0} = \lim_{\varepsilon\to 0} \frac{u(x,t;\varepsilon) - u(x,t;0)}{\varepsilon},\]
\[v^{(1)} := \left. \partial_{\varepsilon} v \right\rvert_{\varepsilon = 0} = \lim_{\varepsilon\to 0} \frac{v(x,t;\varepsilon) - v(x,t;0)}{\varepsilon},\] and consider the new system associated to $(u^{(1)}, v^{(1)})$. 
For $F,G\in C^1$, we have 
\[(u(x,t;0), v(x,t;0))=(0,0),\]
so
\[\partial_t u^{(1)}(x,t) - \mu\Delta u^{(1)}(x,t) = 0\] 
and 
\[\partial_t v^{(1)}(x,t) - \nu\Delta v^{(1)}(x,t) = \beta_{10}(x,t) u^{(1)}(x,t) + \beta_{01}(x,t) v^{(1)}(x,t).\]

Therefore, we have that $\left(u^{(1)},v^{(1)}\right)$ satisfies the following system:
\begin{equation}\label{Linear1}
    \begin{cases}
        \partial_t u^{(1)}(x,t) - \mu\Delta u^{(1)}(x,t) = 0 &\quad \text{in }Q,\\
        \partial_t v^{(1)}(x,t) - \nu\Delta v^{(1)}(x,t) = \beta_{10}(x,t) u^{(1)}(x,t) + \beta_{01}(x,t) v^{(1)}(x,t) &\quad \text{in }Q,\\
        u^{(1)}(x,0) = u_0(x) \geq0,\quad v^{(1)}(x,0) = v_0(x) \geq0  &\quad \text{in }\Omega,\\
        u^{(1)}(x,t) = f_1(x,t) >0,\quad v^{(1)}(x,t) = g(x,t) \geq0 &\quad \text{on }\Sigma.
    \end{cases}
\end{equation}

Then, $u^{(1)}$ satisfies the well-posed heat equation such that $u^{(1)}\in C^{2+\alpha,1+\alpha/2}(\bar{Q})$, which is strictly positive by the maximum principle. Extending $f_1(x,t)$ to $\bar{\Omega}$ which we still denote by $f_1(x,t)$, $\bar{u}:=u^{(1)}-f_1(x,t)$ satisfies the well-posed heat equation
\begin{equation}\label{N1EqU}
    \begin{cases}
        \partial_t \bar{u}(x,t) - \mu\Delta \bar{u}(x,t) = \bar{f}^1(x,t) &\quad \text{in }Q,\\
        \bar{u}(x,0) = \bar{f}^2(x)  &\quad \text{in }\Omega,\\
        \bar{u}_j(x,t) = 0  &\quad \text{on }\Sigma,
    \end{cases}
\end{equation}
where $\bar{f}^1(x,t) \in C^{2+\alpha,1+\alpha/2}(\bar{Q})$, $\bar{f}^2(x) \in C^{2+\alpha}(\bar{\Omega})$ are functions depending on the fixed non-negative known initial value $u_0(x)$ and fixed positive input $f_1(x,t)$.
Then, we can solve \eqref{N1EqU}, and the unique positive solution is given by 
\begin{equation}\label{u1Soln}\bar{u}(x,t) = \int_0^t\int_\Omega\int_\Omega \Phi(x-y-z,t-s) \bar{f}^1(y,s) \bar{f}^2(z) \,dy\,dz\,ds>0,\end{equation}
where $\Phi$ is the fundamental solution of the generalised heat equation on the bounded domain $\Omega$ (given by the mollification with a compactly supported function on $\Omega$).

Next, consider two different values of $\beta_{10}$, given by $\beta_{10}^1$ and $\beta_{10}^2$. Then for $j=1,2$, $v^{(1)}_j$ satisfies 
\begin{equation}\label{N1eqV}\partial_t v^{(1)}_j - \nu\Delta v^{(1)}_j - \beta_{01}(x,t) v^{(1)}_j = \beta_{10}^j u^{(1)}.\end{equation} 
Therefore, $v^{(1)}_j$ is the unique solution given by
\[v^{(1)}_j(x,t) = \int_0^{T-t}\int_\Omega \Psi(x-y,T-t-s) \beta_{10}^j(y,T-s) u^{(1)}(y,T-s) \,dy\,ds,\] where $\Psi$ is the fixed, known Green's function for the operator $\partial_t  - \nu\Delta - \beta_{01}$ on the bounded domain $\Omega$. Note that $v^{(1)}_j$ is not yet determined. But, for $\beta_{10}^j(x,t)\geq0$, $\beta_{01}(x,t)\leq0$, we have that $v^{(1)}_j>0$ by the maximum principle for coercive second order parabolic operators since $u^{(1)}>0$.

\subsection{$m=1$}

We move on to the second order linearisation. Consider
\[u^{(2)}:= \left. \partial_{\varepsilon}^2 u \right\rvert_{\varepsilon = 0}, \quad v^{(2)}:= \left. \partial_{\varepsilon}^2 v \right\rvert_{\varepsilon = 0}.\]
Then, by direct calculations, we have the second order linearisation
\begin{equation}\label{Linear2}
    \begin{cases}
        \partial_t u^{(2)} - \mu\Delta u^{(2)} = 0 &\quad \text{in }Q,\\
        \partial_t v^{(2)} - \nu\Delta v^{(2)} = 2\beta_{20}[u^{(1)}]^2 + 2\beta_{11}u^{(1)}v^{(1)} + 2\beta_{02}[v^{(1)}]^2 + \beta_{10} u^{(2)} + \beta_{01} v^{(2)} &\quad \text{in }Q,\\
        u^{(2)}(x,0) = v^{(2)}(x,0) = 0  &\quad \text{in }\Omega,\\
        u^{(2)}(x,t) = f_2(x,t),\quad v^{(2)}(x,t) = 0  &\quad \text{on }\Sigma.
    \end{cases}
\end{equation}

We have come to our first possible scenario for Theorem \ref{MainThm}:

\begin{theorem}\label{ThmN2}
    Assume that $F,G\in\mathcal{A}$ are such that $\beta_{11}=\beta_{02}\equiv0$, and $\beta_{01},\beta_{20}$ are known and fixed.
    Let $\mathcal{M}^+_{G^j}$ be the measurement map associated to \eqref{MainPDE} for 
    \[u(x,t)=\sum_{l=1}^2\varepsilon^l f_l\quad \text{ on }\Sigma,\] 
    such that Conditions \eqref{Cond1}--\eqref{Cond2} are satisfied, as well as the assumptions of Theorem \ref{MainThmForwardPb}. 
    If 
    \[\mathcal{M}^+_{G^1}(u|_{\partial\Omega})=\mathcal{M}^+_{G^2}(u|_{\partial\Omega}),\]
    then it holds that
    \[\beta_{10}^1(x,t)=\beta_{10}^2(x,t)\text{ in }Q.\]
\end{theorem}

\begin{proof}
From \eqref{Linear2}, $u^{(2)}_j$, $v^{(2)}_j$ satisfy
\begin{equation}\label{Linear2j}
    \begin{cases}
        \partial_t u^{(2)}_j - \mu\Delta u^{(2)}_j = 0 &\quad \text{in }Q,\\
        \partial_t v^{(2)}_j - \nu\Delta v^{(2)}_j = 2\beta_{20}[u^{(1)}]^2 + \beta_{10}^j u^{(2)}_j + \beta_{01} v^{(2)}_j &\quad \text{in }Q,\\
        u^{(2)}_j(x,0) = v^{(2)}_j(x,0) = 0  &\quad \text{in }\Omega,\\
        u^{(2)}_j(x,t) = f_2^j(x,t),\quad v^{(2)}_j(x,t) = 0   &\quad \text{on }\Sigma,
    \end{cases}
\end{equation}
where $u^{(1)}>0$ is the unique solution of \eqref{N1EqU}.

Then once again, $u^{(2)}_j$ satisfies the well-posed heat equation such that $u^{(2)}_j\in C^{2+\alpha,1+\alpha/2}(\bar{Q})$ for $j=1,2$. In this case, since $f_2^j(x,t)$ can be positive or negative, $u^{(2)}_j$ is not strictly positive. Extending the given input $f_2^j(x,t)$ to $\bar{\Omega}$ which we still denote by $f_2^j(x,t)$, the unique solution $u^{(2)}_j$ is given by 
\[u^{(2)}_j(x,t) = \int_0^t\int_\Omega \Phi(x-y,t-s) f^j_2(y,s) \,dy\,ds.\]
When $\mathcal{M}^+_{G^1}=\mathcal{M}^+_{G^2}$, the input data satisfy $f_2^1=f_2^2$, so $u^{(2)}_1=u^{(2)}_2$. 

Next, taking the difference of the two equations for $j=1,2$, we have, denoting $\tilde{v}=v^{(2)}_1-v^{(2)}_2$, 
\begin{align}\partial_t \tilde{v} - \nu\Delta \tilde{v} - \beta_{01} \tilde{v} =  \beta_{10}^1 u^{(2)}_1 - \beta_{10}^2 u^{(2)}_2 &= \beta_{10}^1 (u^{(2)}_1 - u^{(2)}_2) + (\beta_{10}^1 - \beta_{10}^2) u^{(2)}_2  = (\beta_{10}^1 - \beta_{10}^2) u^{(2)}_2. \label{N2eqV}\end{align} Therefore, $\tilde{v}$ is the unique solution given by
\begin{multline*}\tilde{v}(x,t) = \int_0^{T-t}\int_\Omega \Psi(x-y,T-t-s) (\beta_{10}^1(y,T-s) - \beta_{10}^2(y,T-s))  u^{(2)}_2(y,T-s) \,dy\,ds.\end{multline*}

Since $v^{(1)}_j(x,0) = 0$ for $j=1,2$, 
\[\int_0^{T}\int_\Omega \Psi(x-y,T-s) (\beta_{10}^1(y,T-s) - \beta_{10}^2(y,T-s)) u^{(2)}_2(y,T-s) \,dy\,ds = 0,\]
which we write as
\begin{equation}\label{N1eqproof}\int_0^{T} \left[ \Psi * ((\beta_{10}^1 - \beta_{10}^2) u^{(2)}_2) \right] (x,T-s) \,ds = 0.\end{equation} This holds for all $f_2 \in C^{2+\alpha,1+\alpha/2}(\bar{\Sigma})$, hence for all solutions $u^{(2)}_2$ of \eqref{N1EqU}.

Recall that the Fourier transform of a function is given by
\[\mathcal{F}(\varphi)(\psi) := \int_{\mathbb{R}^n} \varphi(x) e^{-2\pi i \xi\cdot x} \,dx,\] where $i:=\sqrt{-1}$, $\xi\in\mathbb{R}^n$. Applying this Fourier transform with respect to the space variable to both sides of \eqref{N1eqproof}, we obtain
\[\int_0^{T} \left[\mathcal{F}(\Psi) \mathcal{F}((\beta_{10}^1 - \beta_{10}^2) u^{(2)}_2) \right] (\xi,T-s) \,ds = 0 \quad \forall \xi\in\mathbb{R}^n\]

Since $G^j$ is continuous with respect to $x$ and $t$, so is $\beta_{10}^j$, so there exists $\hat{\beta}_\eta(t)$ such that
\[\beta_{10}^1(x,t) - \beta_{10}^2(x,t) = \int_{-\infty}^\infty \hat{\beta}_\eta(t) e^{2\pi i \eta\cdot x}.\] 
Choosing $u^{(2)}_2(x,t)$ to be the complex geometric optics (CGO) solution \[u^{(2)}_2(x,t) = e^{-4 \pi^2 |\zeta|^2 t - \frac{2\pi i}{\sqrt{\mu}} \zeta \cdot x }\quad\text{ for any fixed } \zeta\in\mathbb{R}^n,\] we observe that such a $u^{(2)}_2$ satisfies \eqref{Linear2j}. Invoking the Weierstrass' approximation theorem to obtain the density of $\exp \left(-4 \pi^2 |\zeta|^2 t\right)$ in $C^1(0,T)$, we can conclude that $\hat{\beta}_\eta(t)=0$ for all $\eta\in \mathbb{R}^n$. Therefore,
\[\beta_{10}^1(x,t)=\beta_{10}^2(x,t)\quad\forall (x,t)\in Q.\]

\end{proof}

\begin{remark}\label{remark:N2coef}
    The assumption $\beta_{11}=\beta_{02}\equiv0$ is taken since we have not yet determined $v^{(1)}_j$. In particular, $v^{(1)}_j$ depends on $\beta_{10}^j$, which is what we want to recover. We remark that it may be possible to remove this assumption and write out the corresponding equation similar to \eqref{N2eqV} to write out the corresponding identity \eqref{N1eqproof}, so as to argue similarly to obtain the result. However, we have chosen to take this assumption to keep our results clear and emphasise the physical significance of our higher-order variation method in maintaining positivity of solutions.
\end{remark}

Having determined $\beta_{10}^j$, we can now return to the first order linearisation in \eqref{N1eqV} to determine $v^{(1)}_j$. Furthermore, for $\beta_{11},\beta_{02}$ known, fixed and not necessarily equivalent to $0$, we can determine $v^{(2)}_j$, where $v^{(2)}_j$ satisfies
\begin{equation}\label{N2EqVSolve}\partial_t v^{(2)} - \nu\Delta v^{(2)} - \beta_{01} v^{(2)} = 2\beta_{20}[u^{(1)}]^2 + 2\beta_{11}u^{(1)}v^{(1)} + 2\beta_{02}[v^{(1)}]^2 + \beta_{10} u^{(2)} \quad \text{in }Q.\end{equation} In particular, setting 
\[\mathbb{V}^{(2)}(x,t):=2\beta_{20}[u^{(1)}]^2 + 2\beta_{11}u^{(1)}v^{(1)} + 2\beta_{02}[v^{(1)}]^2, \]
we have 
\begin{equation}\label{v2soln}v^{(2)}_j(x,t) = \int_0^{T-t}\int_\Omega \Psi(x-y,T-t-s) [\mathbb{V}^{(2)} + \beta_{10}^j u^{(2)}_j](y,T-s) \,dy\,ds.\end{equation}

\subsection{$m=2$}

We next show the result for $m=2$, by considering the third order linearisation.

Defining $u^{(3)}, v^{(3)}$ inductively, we have the third order linearisation:
\begin{equation}\label{Linear3}
    \begin{cases}
        \partial_t u^{(3)} - \mu\Delta u^{(3)} = 6\alpha_{30}[u^{(1)}]^3 + 6\alpha_{03}[v^{(1)}]^3 + 6\alpha_{12}u^{(1)}[v^{(1)}]^2 + 6\alpha_{21}[u^{(1)}]^2v^{(1)} &\quad \text{in }Q,\\
        \partial_t v^{(3)} - \nu\Delta v^{(3)} = 6\beta_{30}[u^{(1)}]^3 + 6\beta_{03}[v^{(1)}]^3 + 6\beta_{12}u^{(1)}[v^{(1)}]^2 + 6\beta_{21}[u^{(1)}]^2v^{(1)} \\ 
        \quad\quad\quad\quad\quad\quad\quad\quad + 6\beta_{20}u^{(1)}u^{(2)} + 6\beta_{02}v^{(1)}v^{(2)} + 3\beta_{11}u^{(2)}v^{(1)} \\ 
        \quad\quad\quad\quad\quad\quad\quad\quad + 3\beta_{11}u^{(1)}v^{(2)}  + \beta_{10} u^{(3)} + \beta_{01} v^{(3)} &\quad \text{in }Q,\\
        u^{(3)}(x,0) = v^{(3)}(x,0) = 0  &\quad \text{in }\Omega,\\
        u^{(3)}(x,t) = f_3(x,t),\quad v^{(3)}(x,t) = 0  &\quad \text{on }\Sigma.
    \end{cases}
\end{equation}

Our next scenario for Theorem \ref{MainThm} is as follows:

\begin{theorem}\label{ThmN3}
    Assume that $F,G\in\mathcal{A}$ are such that $\beta_{01}(x,t)\leq0$, \[\beta_{11}=\beta_{02}\equiv0,\] and all the remaining coefficients of \eqref{Linear3} are known and fixed except for $\beta_{20}$.
    Let $\mathcal{M}^+_{F^j,G^j}$ be the measurement map associated to \eqref{MainPDE} for 
    \[u(x,t)=\sum_{l=1}^3\varepsilon^l f_l\quad \text{ on }\Sigma,\] 
    such that Conditions \eqref{Cond1}--\eqref{Cond2} are satisfied, as well as the assumptions of Theorem \ref{MainThmForwardPb}. 
    If 
    \[\mathcal{M}^+_{F^1,G^1}(u|_{\partial\Omega})=\mathcal{M}^+_{F^2,G^2}(u|_{\partial\Omega}),\]
    then it holds that
    \[\beta_{20}^1(x,t)=\beta_{20}^2(x,t)\quad\text{ in }Q.\]

\end{theorem}


\begin{proof}
For input $f$ and given functions $u_0,v_0,g$ satisfying the assumptions of Theorem \ref{MainThmForwardPb}, since all the coefficients of \eqref{Linear3} are known and fixed except for $\beta_{20}$, we can compute $u^{(1)}$, $u^{(2)}_j$ and $v^{(1)}$ (given $\beta_{10}$ known, fixed) using \eqref{Linear1} and \eqref{Linear2j}. Consequently, since all the coefficients $\alpha_{mn}$, $m+n=3$, are fixed and known, we first compute $u^{(3)}$ from \eqref{Linear3} and obtain a solution $u^{(3)}\in C^{2+\alpha,1+\alpha/2}(\bar{Q})$. Once again, $u^{(3)}$ may be positive or negative.

Then, again taking the difference of the two equations for the corresponding $j$-th problem of \eqref{Linear3} for $j=1,2$, we have, denoting $\tilde{v}=v^{(3)}_1-v^{(3)}_2$, 
\begin{equation}\label{N3eqV}\partial_t \tilde{v} - \nu\Delta \tilde{v} - \beta_{01}\tilde{v} =  6(\beta_{20}^1 - \beta_{20}^2) u^{(1)} u^{(2)}_2\end{equation}
when $\mathcal{M}^+_{F^1,G^1}=\mathcal{M}^+_{F^1,G^2}$.

Applying the same argument as in the case for $m=1$ to \eqref{N3eqV} with the same CGO solution for $u^{(2)}_2$, we have that 
\[(\beta_{20}^1 - \beta_{20}^2) u^{(1)} = 0 \quad\forall (x,t)\in Q.\] Since $u^{(1)}$ satisfies the parabolic equation \eqref{Linear1} with positive initial and boundary conditions, by the maximum principle, $u^{(1)}>0$ for all $x,t$, and we conclude that 
\[\beta_{20}^1(x,t)=\beta_{20}^2(x,t)\quad\text{ in }Q.\]

\end{proof}

\begin{remark}\label{remark:N3coef}
    The assumption $\beta_{01}(x,t)\leq0$ is to ensure coercivity of the operator $\partial_t-\Delta-\beta_{10}$. 

    The assumption that the coefficients $\alpha_{mn}$, $m+n=3$, are fixed and known allows us to compute $u^{(3)}=u^{(3)}_1=u^{(3)}_2$. This is natural since we input a fixed boundary condition $f_3$, and the other involved terms $u^{(1)}$, $v^{(1)}$ are fixed. One may ask whether it is possible to determine the coefficients $\alpha_{mn}$. We think that it may be possible to separately determine each of the coefficients $\alpha_{mn}$ assuming the others are fixed, by multiplying \eqref{Linear3} by a test function $\psi$ and taking suitable $\psi$ to obtain the result. One may refer to the proofs given in \cite{LiuZhang2022-InversePbMFG} and \cite{LiuZhangMFG3} for more details, where one chooses $\psi$ such that the chosen functions span the solution space of the adjoint equation. However, this method does not ensure the positivity of solutions, which is the core concept of this study, so we choose to leave it out here.

    A similar argument applies for the coefficients $\beta_{mn}$, $m+n=3$, as well as $\beta_{10}$.

    Finally, we discuss the assumption 
    \[\beta_{11}=\beta_{02}\equiv0.\] We recall that $v^{(2)}_j$ varies together with $u^{(2)}_j$ and $\beta_{20}^j$, via the equation \eqref{Linear2}. Following the discussion given in Remark \ref{remark:N2coef}, it may be possible to consider \eqref{N3eqV} with nonzero (and possibly non-fixed) coefficients $\beta_{11}$, $\beta_{02}$ to obtain a more complicated identity similar to \eqref{N1eqproof} for the case of the third order linearisation, and obtain the same result, or even the equality of $\beta_{11}^j$ or $\beta_{02}^j$ for $j=1,2$. Therefore, our choice of this assumption is only one possible choice that is sufficient to obtain the result, but is not a necessary condition.
\end{remark}

\subsection{$m\geq3$}

We will first show the results for $m=3$. Defining $u^{(4)}, v^{(4)}$ inductively, we have the fourth order linearisation:
\begin{equation}\label{Linear4}
    \begin{cases}
        \partial_t u^{(4)} - \mu\Delta u^{(4)} = 24\alpha_{40}[u^{(1)}]^4 + 24\alpha_{04}[v^{(1)}]^4 + 6\alpha_{31}[u^{(1)}]^3v^{(1)} + 6\alpha_{13}u^{(1)}[v^{(1)}]^3 \\ 
        \quad\quad\quad\quad\quad\quad\quad\quad  + 2\alpha_{22}[u^{(1)}]^2[v^{(1)}]^2+ 18\alpha_{30}[u^{(1)}]^2u^{(2)} + 18\alpha_{03}[v^{(1)}]^2v^{(2)}\\ 
        \quad\quad\quad\quad\quad\quad\quad\quad  + 6\alpha_{12}u^{(2)}[v^{(1)}]^2  + 12\alpha_{12}u^{(1)}v^{(1)}v^{(2)} + 12\alpha_{21}u^{(1)}u^{(2)}v^{(1)} \\ 
        \quad\quad\quad\quad\quad\quad\quad\quad  + 6\alpha_{21}[u^{(1)}]^2v^{(2)} &\quad \text{in }Q,\\
        \partial_t v^{(4)} - \nu\Delta v^{(4)} = 24\beta_{40}[u^{(1)}]^4 + 24\beta_{04}[v^{(1)}]^4 + 6\beta_{31}[u^{(1)}]^3v^{(1)} + 6\beta_{13}u^{(1)}[v^{(1)}]^3 \\ 
        \quad\quad\quad\quad\quad\quad\quad\quad  + 2\beta_{22}[u^{(1)}]^2[v^{(1)}]^2 + 18\beta_{30}[u^{(1)}]^2u^{(2)} + 18\beta_{03}[v^{(1)}]^2v^{(2)}\\ 
        \quad\quad\quad\quad\quad\quad\quad\quad  + 6\beta_{12}u^{(2)}[v^{(1)}]^2  + 12\beta_{12}u^{(1)}v^{(1)}v^{(2)} + 12\beta_{21}u^{(1)}u^{(2)}v^{(1)} \\ 
        \quad\quad\quad\quad\quad\quad\quad\quad  + 6\beta_{21}[u^{(1)}]^2v^{(2)} + 6\beta_{20}u^{(1)}u^{(3)} + 6\beta_{20}[u^{(2)}]^2\\
        \quad\quad\quad\quad\quad\quad\quad\quad + 6\beta_{02}v^{(1)}v^{(3)} + 6\beta_{02}[v^{(2)}]^2 +
        6\beta_{11}u^{(2)}v^{(2)} +
        3\beta_{11}u^{(3)}v^{(1)} \\ 
        \quad\quad\quad\quad\quad\quad\quad\quad  + 3\beta_{11}u^{(1)}v^{(3)} + \beta_{10} u^{(4)} + \beta_{01} v^{(4)} + \beta_{00} &\quad \text{in }Q,\\
        u^{(4)}(x,0) = v^{(4)}(x,0) = 0  &\quad \text{in }\Omega,\\
        u^{(4)}(x,t) = f_4(x,t),\quad v^{(4)}(x,t) = 0  &\quad \text{on }\Sigma.
    \end{cases}
\end{equation}

Then, our first result for $m=3$ corresponding to Theorem \ref{MainThm}(2) reads:

\begin{theorem}\label{ThmN4U}
    Assume that $F,G\in\mathcal{A}$ are such that \[\alpha_{03}=\alpha_{12}=\alpha_{21}\equiv0,\] and $\beta_{10}$ and $\alpha_{m_2n_2}$ are known and fixed for $m_2+n_2=4$. Let $\mathcal{M}^+_{F^j,G^j}$ be the measurement map associated to \eqref{MainPDE} for 
    \[u(x,t)=\sum_{l=1}^4\varepsilon^l f_l\quad \text{ on }\Sigma,\] 
    such that Conditions \eqref{Cond1}--\eqref{Cond2} are satisfied, as well as the assumptions of Theorem \ref{MainThmForwardPb}. 
    If 
    \[\mathcal{M}^+_{F^1,G^1}(u|_{\partial\Omega})=\mathcal{M}^+_{F^2,G^2}(u|_{\partial\Omega}),\]
    then it holds that
    \[\alpha_{30}^1(x,t)=\alpha_{30}^2(x,t)\quad\text{ in }Q.\]
\end{theorem}

\begin{proof}
As in the proof of the previous Theorem \ref{ThmN3}, 
for input $f$ and given functions $u_0,v_0,g$ satisfying the assumptions of Theorem \ref{MainThmForwardPb}, since $\beta_{10}$ is known and fixed, we can compute $u^{(1)}$, $u^{(2)}_j$ and $v^{(1)}$, using \eqref{Linear1} and \eqref{Linear2j}. 

Taking the difference of the two equations for the corresponding $j$-th problem of \eqref{Linear4} for $j=1,2$, we have, denoting $\tilde{u}=u^{(4)}_1-u^{(4)}_2$, 
\begin{equation}\label{N4eqU}
 \partial_t \tilde{u} - \mu\Delta \tilde{u} = 18(\alpha_{30}^1 - \alpha_{30}^2) [u^{(1)}]^2u^{(2)}_2
\end{equation}
when $\mathcal{M}^+_{F^1,G^1}=\mathcal{M}^+_{F^2,G^2}$.

Applying the same argument as in the case for $m=1$ to \eqref{N4eqU} with the CGO solution again for $u^{(1)}_2$, we have that 
\[(\alpha_{30}^1(x,t) - \alpha_{30}^2(x,t)) [u^{(1)}(x,t)]^2 = 0 \quad\forall (x,t)\in Q.\] Now, recall once again that $u^{(1)}$ is given by the heat equation \eqref{N1EqU} with positive boundary condition and non-negative initial condition, so it is unique and strictly positive and given by $\bar{u}+f_1(x,t)$ for $\bar{u}$ given by \eqref{u1Soln}. Therefore, 
\[\alpha_{30}^1(x,t)=\alpha_{30}^2(x,t)\quad\text{ in }Q.\]

\end{proof}

\begin{remark}
    Once again, a similar argument as in Remark \ref{remark:N2coef} and the last part of Remark \ref{remark:N3coef} can be applied to this result.
\end{remark}

We also obtain another result for $m=3$, corresponding to Theorem \ref{MainThm}(3):

\begin{theorem}\label{ThmN4V}
    
    Assume that $F,G\in\mathcal{A}$ are such that $\beta_{01}(x,t)\leq0$, $\alpha_{mn}$ are known and fixed for all $m,n$, and 
    \[\beta_{11}=\beta_{20}=\beta_{02}=\beta_{12}=\beta_{21}=\beta_{03}\equiv0\] and $\beta_{m_2n_2}$ are known and fixed for $m_2+n_2=4$, as well as $\beta_{10},\beta_{01}$. 
    Suppose 
    \[\text{either }\quad\beta_{10}\equiv0\quad\text{ or }\quad\alpha_{30}=\alpha_{12}=\alpha_{21}=\alpha_{03}\equiv0.\] 
    Let $\mathcal{M}^+_{F^j,G^j}$ be the measurement map associated to \eqref{MainPDE} for 
    \[u(x,t)=\sum_{l=1}^4\varepsilon^l f_l\quad \text{ on }\Sigma,\] 
    such that Conditions \eqref{Cond1}--\eqref{Cond2} are satisfied, as well as the assumptions of Theorem \ref{MainThmForwardPb}. 
    If 
    \[\mathcal{M}^+_{F^1,G^1}(u|_{\partial\Omega})=\mathcal{M}^+_{F^2,G^2}(u|_{\partial\Omega}),\]
    then it holds that
    \[\beta_{30}^1(x,t)=\beta_{30}^2(x,t)\quad\text{ in }Q.\]
\end{theorem}

\begin{proof}
Once again, for input $f$ and given functions $u_0,v_0,g$ satisfying the assumptions of Theorem \ref{MainThmForwardPb}, since $\beta_{10}$ is known and fixed, we can compute $u^{(1)}$, $u^{(2)}_j$ and $v^{(1)}$, using \eqref{Linear1} and \eqref{Linear2j}. 

Taking the difference of the two equations for the corresponding $j$-th problem of \eqref{Linear4} for $j=1,2$, we have, denoting $\tilde{u}=u^{(4)}_1-u^{(4)}_2$ and similarly $\tilde{v}=v^{(4)}_1-v^{(4)}_2$, 
\begin{equation}\label{N4eqV}
    \partial_t \tilde{v} - \nu\Delta \tilde{v} - \beta_{01} \tilde{v} = 18(\beta_{30}^1-\beta_{30}^2)[u^{(1)}]^2u^{(2)}_2 + \beta_{10} \tilde{u}
\end{equation}
when $\mathcal{M}^+_{F^1,G^1}=\mathcal{M}^+_{F^2,G^2}$. Observe that \eqref{N4eqV} is different from \eqref{N3eqV} with an additional term involving $\tilde{u}$. Note that even when $\mathcal{M}^+_{F^1,G^1}=\mathcal{M}^+_{F^2,G^2}$, we are unable to ensure that $u^{(4)}_1 = u^{(4)}_2$, unless when the assumptions of Theorem \ref{ThmN4U} are satisfied, i.e. 
\[\alpha_{12}=\alpha_{21}=\alpha_{03}\equiv0.\] Alternatively, \eqref{N4eqV} takes the same form as \eqref{N3eqV} when $\beta_{10}\equiv0$. In either of these two cases which are stated as assumptions of our Theorem, \eqref{N4eqV} reduces to 
\begin{equation}
    \partial_t \tilde{v} - \nu\Delta \tilde{v} - \beta_{01} \tilde{v} = 18(\beta_{30}^1-\beta_{30}^2)[u^{(1)}]^2u^{(2)}_2
\end{equation}

Then, applying the same argument as in the case for $m=1$ to \eqref{N4eqU} with the CGO solution again for $u^{(1)}_2$, we have that 
\[(\beta_{30}^1(x,t) - \beta_{30}^2(x,t)) [u^{(1)}(x,t)]^2 = 0 \quad\forall (x,t)\in Q.\] As in the previous Theorem \ref{ThmN4U}, since $u^{(1)}$ is strictly positive, we have the result 
\[\beta_{30}^1(x,t)=\beta_{30}^2(x,t)\quad\text{ in }Q.\]

\end{proof}

\begin{remark}
    Once again, a similar argument as in Remark \ref{remark:N3coef} can be applied to this result.
\end{remark}

\begin{proof}[Proof of Main Theorem \ref{MainThm}]
    We conclude the proof of Theorem \ref{MainThm} by considering  \[u(x,t)=\sum_{l=1}^{m+1}\varepsilon^l f_l\quad \text{ on }\Sigma,\] and the cases for $m\geq 3$ follow first similarly for $u$, as in Theorem \ref{ThmN4U}, and we obtain the result for Theorem \ref{MainThm}(2). Having obtained the equality of $u^{(m+1)}_j$ when $\alpha_{m_1n_1}\equiv0$ for all $2\leq m_1+n_1\leq m$, or when $\beta_{10}\equiv0$ such that the only terms involved are $u^{(1)}_j,v^{(1)}_j$ which are equal for $j=1,2$ when $\mathcal{M}^+_{F^1,G^1}=\mathcal{M}^+_{F^2,G^2}$, we can also argue similarly for $v$ to obtain the result for $\beta_{m0}^j$.
\end{proof}

\begin{remark}
This result can be easily extended to general second order parabolic operators of the form $\partial_t - \nabla \cdot (\sigma \nabla)$ for some fixed known measurable, bounded, coercive matrix $\sigma(x)$, by using the results of \cite{CGOGeneralLap}. Examples of such anisotropic PDE equations/systems include the anisotropic Fokker-Planck equation (see, for instance, \cite{AnisotropicFokkerPlanck1,AnisotropicFokkerPlanck2}), or PDE systems modelling general population dynamics (see, for instance, \cite{PredatorPreyAnisotropic1} or the discussion in the Introduction of \cite{PredatorPreyAnisotropic2}).
\end{remark}

\section{Biological Applications: Reactive-Diffusive Predator-Prey Models}\label{sect:app}
Our results can be applied to a variety of models, some examples of which have been given in Section \ref{sect:motivation}. Here, we explain our results by describing a group of examples modelling the dynamics of ecological differential systems with self diffusion given by diffusion constants $\mu,\nu>0$, for the case where the densities of the predator and prey are spatially inhomogeneous in a bounded domain  $\Omega$ with smooth boundary, of the form:
\begin{equation}\label{ModelGeneral}
    \begin{cases}
        \partial_t u - \mu\Delta u = F(u,v) &\quad \text{in }Q,\\
        \partial_t v - \nu\Delta v = G(u,v) &\quad \text{in }Q,\\
        u,v\geq0    &\quad \text{in }Q,\\
        u(x,0)=u_0(x)\geq0,\quad v(x,0) = v_0(x)\geq0  &\quad \text{in }\Omega, \\
        u = f\geq0,\quad v = g\geq0 &\quad \text{on }\Sigma.
    \end{cases}
\end{equation}
Here, $u$ and $v$ represent the prey population density $U$ and predator population density  $V$ respectively, such that the initial prey and predator population densities $u_0(x)>0$ are non-negative. Furthermore, the prey and predator population densities $f,g$ on the boundary $\partial\Omega$ of a Lipschitz bounded $\Omega$ for all time $t\in(0,T)$ are non-negative. 

Now, suppose the prey population $U$ adopts a cubic growth profile $au^3$ with constant $a>0$ (see, for instance, \cite{PredPreyCubic,PredPreyCubic2,PredPreyCubicDiffusive}), whereas the predator population $V$ has a growth rate $\alpha,b,\gamma>0$ depending on the prey population, as well as death rates $c,\beta>0$, so that the growth of the predator population $V$ could be represented by $bu-cv+(\alpha u - \beta v)v$. Finally, we assume that the predator adopts higher order cooperation in hunting, as in \cite{PredPreyHunting,PredPreyHuntingDiffusive}, so that the predation rate of the predator is given by terms of the form $\gamma uv$ and $(\lambda + \mu v) u^2$. Then, the functions $F$ and $G$ in \eqref{ModelGeneral} are specifically given by 
\[F(u,v)=a u^3 - (\lambda + \mu v) u^2 v,\]
and 
\[G(u,v)=bu - cv + (\alpha u - \beta v + \gamma uv) v + (\lambda + \mu v) u^2 v.\]

Substituting this into \eqref{ModelGeneral}, the reactive-diffusive predator-prey model we are considering is given by
\begin{equation}\label{Model}
    \begin{cases}
        \partial_t u - \mu\Delta u = a u^3 - (\lambda + \mu v) u^2 v &\quad \text{in }Q,\\
        \partial_t v - \nu\Delta v = bu - cv + (\alpha u - \beta v + \gamma uv) v + (\lambda + \mu v) u^2 v &\quad \text{in }Q,\\
        u,v\geq0    &\quad \text{in }Q,\\
        u(x,0)=u_0(x)\geq0,\quad v(x,0) = v_0(x)\geq0  &\quad \text{in }\Omega, \\
        u = f\geq0,\quad v = g\geq0 &\quad \text{on }\Sigma.
    \end{cases}
\end{equation}
Clearly, the origin $(0,0)$ is an equilibrium point for this system. 

Then, we have two possible results, the first of which is given by the following:

\begin{proposition}
Assume that $\mu,\nu>0$ and $a(x,t)$, $b(x,t)$, $c(x,t)$, $\alpha(x,t)$, $\beta(x,t)$, $\gamma(x,t)$, $\lambda(x,t)$, $\mu(x,t)$ are continuous with respect to $x$ and $t$. Suppose \[\alpha(x,t)=\beta(x,t)\equiv0\] and $c(x,t)$ known and fixed for each $x,t$. Let $\mathcal{M}^+_{G^j}$ be the measurement map associated to \eqref{Model} for 
\[u(x,t)=\sum_{l=1}^2\varepsilon^l f_l\quad \text{ on }\Sigma,\] 
such that Conditions \eqref{Cond1}--\eqref{Cond2} are satisfied, as well as the assumptions of Theorem \ref{MainThmForwardPb}. In particular, this means that 
\[f_1(x,t)>0\quad\forall x\in\Omega,t\in(0,T).\] 
If 
\[\mathcal{M}^+_{G^1}(u|_{\partial\Omega})=\mathcal{M}^+_{G^2}(u|_{\partial\Omega}),\]
then it holds that
\[b^1(x,t)=b^2(x,t)\text{ in }Q.\]
\end{proposition}

\begin{proof}
    This follows simply by applying Theorem \ref{ThmN2}.
\end{proof}

The assumption on $f_1(x,t)>0$ means that the prey population density at the boundary $\partial\Omega$ is strictly positive for all time $t\in(0,T)$, though it is close to 0. 
Physically, our result can be interpreted in the sense that, we input a certain population of prey from the boundary into a bounded encircled region $\Omega$, and measure the prey population density value as well as its flux at the boundary of $\Omega$. If all these values are the same, and assuming a predator growth rate that consists of a term of zeroth order with respect to the prey, as well as terms of at least cubic order, and that the linear death rate of the predator population is known and fixed, there is only one possible value for predator's zeroth order dependence growth coefficient with respect to the prey.

On the other hand, we can also apply Theorem \ref{ThmN4U} instead to obtain the second possible result:
\begin{proposition}
Assume that $\mu,\nu>0$ and $a(x,t)$, $b(x,t)$, $c(x,t)$, $\alpha(x,t)$, $\beta(x,t)$, $\gamma(x,t)$, $\lambda(x,t)$, $\mu(x,t)$ are continuous with respect to $x$ and $t$. Suppose \[\lambda(x,t)\equiv0.\]  Let $\mathcal{M}^+_{F^j,G^j}$ be the measurement map associated to \eqref{Model} for 
\[u(x,t)=\sum_{l=1}^4\varepsilon^l f_l\quad \text{ on }\Sigma,\] 
such that Conditions \eqref{Cond1}--\eqref{Cond2} are satisfied, as well as the assumptions of Theorem \ref{MainThmForwardPb}. In particular, this once again means that 
\[f_1(x,t)>0\quad\forall x\in\Omega,t\in(0,T).\]
If 
\[\mathcal{M}^+_{F^1,G^1}(u|_{\partial\Omega})=\mathcal{M}^+_{F^2,G^2}(u|_{\partial\Omega}),\]
then it holds that
\[a^1(x,t)=a^2(x,t) \quad \text{ in }Q.\]
\end{proposition}

\begin{proof}
    This simply follows from Theorem \ref{ThmN4U} for the case $m=3$.
\end{proof}

Once again, we have assumed that the prey population density at the boundary $\partial\Omega$ is strictly positive for all time $t\in(0,T)$, by assuming $f_1(x,t)>0$. In the physical setting, this result means that we input a certain population of prey from the boundary into a bounded encircled region $\Omega$, and measure the prey population density value as well as its flux at the boundary of $\Omega$. If all these values are the same, there is only one possible value for the cubic prey growth rate.

Note that with the same model system \eqref{Model}, two different results are obtained under two different set of assumptions. This is an interesting aspect of our result, and arises only when we needed to enforce the positivity requirement.

\section{Discussion of Results}\label{sect:discuss}
\subsection{Technical Challenges}
Our main novelty in this paper is to consider \[u(x,t;\varepsilon)=\sum_{l=1}^\infty\varepsilon^l f_l\quad \text{ on }\Sigma\quad \text{ for }f_1>0,\] with $f_2(x,t)$ possibly positive or negative at different $x,t$, so that $u(x,t;\varepsilon)>0$ on $\Sigma$ for all small positive $\varepsilon$. In this regard, it can be seen, from Theorem \ref{MainThm}, that we can only derive the coefficient for $u^m$ for $m\geq1$, and in addition, with many assumptions required, such as the coefficients of terms of lower and equal order being zero. Although such assumptions are by no means restrictive, we can notice from the previous section that this still limits the types of physical models we can apply our method to (for instance, no linear nor quadratic terms for the first variable), as well as the amount of information on the involved parameters that we can recover via an input of the boundary data.

This limitation arises from the enforcement of positivity of the solutions. Indeed, we see that if positivity is not required, the boundary data of $u$ and $v$ can be used to fully determine the semilinear terms $F$ and $G$, for all coefficients $\alpha_{mn},\beta_{mn}$ for any $m,n\geq0$, $m+n\geq1$. This can be easily done by viewing the system of equations \eqref{MainPDE} as a single parabolic equation for $w=(u,v)$, and applying the result in \cite{LinLiuLiuZhang2021-InversePbSemilinearParabolic-CGOSolnsSuccessiveLinearisation} given as Theorem 1.3, by once again making use of the same method of higher order linearisation we have used here.

At the same time, it should be noted that in our case, our measurement map only involves $u$, and no information is required for $v$. This is different from that considered in \cite{LinLiuLiuZhang2021-InversePbSemilinearParabolic-CGOSolnsSuccessiveLinearisation}. As such, our result can, in fact, be viewed as an inverse problem involving partial data.

\subsection{Future Outlook}
Despite these technical challenges, the notion of enforcing the positivity requirement is still crucial, particularly in many physical applications, such as in the model example \eqref{Model} we have considered. Furthermore, many previous works on the forward problem for parabolic PDEs require the positivity of solutions to obtain further properties including uniqueness, regularity and maximum principles (see, for instance, \cite{Positive2,PaoNonlinearBook,Positive1}). However, to the best of our knowledge, our method is one of the only two methods available to consider positive solutions, the other being expansion around $1$ as done in \cite{LiuZhang2022-InversePbMFG}. Our next step will be to consider that possibility instead for the ecological models similar to \eqref{Model}, and discover the parameters we can identify in that case, which may be similar or different from what we obtain here. We choose to investigate these extensions in our future works, and end with the highly intriguing challenge to readers to identify other possible methods to consider the inverse problem for PDEs with positive solutions.

	\medskip 
	
	\noindent\textbf{Acknowledgment.} 
	H. Liu is supported by the Hong Kong RGC General Research Fund (projects 12302919, 12301218 and 11300821) and the NSFC/RGC Joint Research Grant (project N\_CityU101/21).

\bibliographystyle{plain}
\bibliography{ref.bib,ref1.bib,refParaSysInverse.bib}

\begin{thebibliography}{10}

\bibitem{ACM2022}
El~Mustapha Ait Ben~Hassi, Salah-Eddine Chorfi, and Lahcen Maniar.
\newblock Stable determination of coefficients in semilinear parabolic system
  with dynamic boundary conditions.
\newblock {\em Inverse Problems}, 38(11):Paper No. 115007, 28, 2022.

\bibitem{BCGPY2009}
Assia Benabdallah, Michel Cristofol, Patricia Gaitan, and Masahiro Yamamoto.
\newblock Inverse problem for a parabolic system with two components by
  measurements of one component.
\newblock {\em Appl. Anal.}, 88(5):683--709, 2009.

\bibitem{PredatorPreyAnisotropic1}
Mostafa Bendahmane, Michel Langlais, and Mazen Saad.
\newblock On some anisotropic reaction-diffusion systems with {$L^1$}-data
  modeling the propagation of an epidemic disease.
\newblock {\em Nonlinear Anal.}, 54(4):617--636, 2003.

\bibitem{PredPreyCubicDiffusive}
Qunyi Bie.
\newblock Qualitative analysis on a cubic predator-prey system with diffusion.
\newblock {\em Electronic Journal of Qualitative Theory of Differential
  Equations}, 2011, 04 2011.

\bibitem{BFM2014}
Idriss Boutaayamou, Genni Fragnelli, and Lahcen Maniar.
\newblock Lipschitz stability for linear parabolic systems with interior
  degeneracy.
\newblock {\em Electron. J. Differential Equations}, pages No. 167, 26, 2014.

\bibitem{PredPreyCubic2}
Huaihuo Cao and Shengmao Fu.
\newblock Global existence and convergence of solutions to a cross-diffusion
  cubic predator-prey system with stage structure for the prey.
\newblock {\em Bound. Value Probl.}, pages Art. ID 285961, 24, 2010.

\bibitem{CGOGeneralLap}
Pedro Caro and Yavar Kian.
\newblock Determination of convection terms and quasi-linearities appearing in
  diffusion equations.
\newblock {\em arXiv}, 2018.

\bibitem{CGR2006}
Michel Cristofol, Patricia Gaitan, and Hichem Ramoul.
\newblock Inverse problems for a {$2\times 2$} reaction-diffusion system using
  a {C}arleman estimate with one observation.
\newblock {\em Inverse Problems}, 22(5):1561--1573, 2006.

\bibitem{FokkerPlanckNoise}
G.~Da~Prato, F.~Flandoli, and M.~R\"{o}ckner.
\newblock Fokker-{P}lanck equations for {SPDE} with non-trace-class noise.
\newblock {\em Commun. Math. Stat.}, 1(3):281--304, 2013.

\bibitem{AnisotropicFokkerPlanck1}
P.C. {da Silva}, L.R. {da Silva}, E.K. Lenzi, R.S. Mendes, and L.C. Malacarne.
\newblock Anomalous diffusion and anisotropic nonlinear {F}okker-{P}lanck
  equation.
\newblock {\em Phys. A: Stat. Mech. Appl.}, 342(1):16--21, 2004.
\newblock Proceedings of the VIII Latin American Workshop on Nonlinear
  Phenomena.

\bibitem{Positive2}
Avner Friedman.
\newblock {\em Partial differential equations of parabolic type}.
\newblock Prentice-Hall, Inc., Englewood Cliffs, N.J., 1964.

\bibitem{AnisotropicFokkerPlanck2}
Maxime Herda and Luis~Miguel Rodrigues.
\newblock Anisotropic {B}oltzmann-{G}ibbs dynamics of strongly magnetized
  {V}lasov-{F}okker-{P}lanck equations.
\newblock {\em Kinet. Relat. Models}, 12(3):593--636, 2019.

\bibitem{PredatorPreyAnisotropic2}
Dirk Horstmann.
\newblock Remarks on some {L}otka-{V}olterra type cross-diffusion models.
\newblock {\em Nonlinear Anal. Real World Appl.}, 8(1):90--117, 2007.

\bibitem{PredPreyCubic}
Xuncheng Huang, Yuanming Wang, and Lemin Zhu.
\newblock One and three limit cycles in a cubic predator-prey system.
\newblock {\em Math. Methods Appl. Sci.}, 30(5):501--511, 2007.

\bibitem{IY1998}
Oleg~Yu. Imanuvilov and Masahiro Yamamoto.
\newblock Lipschitz stability in inverse parabolic problems by the {C}arleman
  estimate.
\newblock {\em Inverse Problems}, 14(5):1229--1245, 1998.

\bibitem{LadyzhenskayaSolonnikovUraltsevaBook}
O.~A. Lady\v{z}enskaja, V.~A. Solonnikov, and N.~N. Ural'ceva.
\newblock {\em Linear and quasilinear equations of parabolic type}.
\newblock Translations of Mathematical Monographs, Vol. 23. American
  Mathematical Society, Providence, R.I., 1968.
\newblock Translated from the Russian by S. Smith.

\bibitem{LassasLiimatainenLinSalo2021InversePbEllipticPowerNonlinear}
Matti Lassas, Tony Liimatainen, Yi-Hsuan Lin, and Mikko Salo.
\newblock Inverse problems for elliptic equations with power type
  nonlinearities.
\newblock {\em J. Math. Pures Appl. (9)}, 145:44--82, 2021.

\bibitem{LassasLiimatainenLinSalo2021InversePbEllipticSemilinear}
Matti Lassas, Tony Liimatainen, Yi-Hsuan Lin, and Mikko Salo.
\newblock Partial data inverse problems and simultaneous recovery of boundary
  and coefficients for semilinear elliptic equations.
\newblock {\em Rev. Mat. Iberoam.}, 37(4):1553--1580, 2021.

\bibitem{LiimatainenLin2022-SemilinearClassicalInverseProblem}
Tony Liimatainen and Yi-Hsuan Lin.
\newblock Uniqueness results and gauge breaking for inverse source problems of
  semilinear elliptic equations.
\newblock {\em arXiv: 2204.11774}, 2023.

\bibitem{LinLiuLiuZhang2021-InversePbSemilinearParabolic-CGOSolnsSuccessiveLinearisation}
Yi-Hsuan Lin, Hongyu Liu, Xu~Liu, and Shen Zhang.
\newblock Simultaneous recoveries for semilinear parabolic systems.
\newblock {\em Inverse Problems}, 38(11):Paper No. 115006, 39, 2022.

\bibitem{LiuMouZhang2022-InverseProblemsMeanFieldGames}
Hongyu Liu, Chenchen Mou, and Shen Zhang.
\newblock Inverse problems for mean field games.
\newblock {\em arXiv: 2205.11350}, 2022.

\bibitem{LiuZhang2022-InversePbMFG}
Hongyu Liu and Shen Zhang.
\newblock On an inverse boundary problem for mean field games.
\newblock {\em arXiv}, 2022.

\bibitem{LiuZhangMFG3}
Hongyu Liu and Shen Zhang.
\newblock Simultaneously recovering running cost and {H}amiltonian in mean field
  games system.
\newblock {\em arXiv: 2303.13096}, 2023.

\bibitem{PaoNonlinearBook}
C.~V. Pao.
\newblock {\em Nonlinear parabolic and elliptic equations}.
\newblock Plenum Press, New York, 1992.

\bibitem{Positive1}
Michel Pierre.
\newblock Global existence in reaction-diffusion systems with control of mass:
  a survey.
\newblock {\em Milan J. Math.}, 78(2):417--455, 2010.

\bibitem{PredPreyHuntingDiffusive}
Danxia Song, Chao Li, and Yongli Song.
\newblock Stability and cross-diffusion-driven instability in a diffusive
  predator-prey system with hunting cooperation functional response.
\newblock {\em Nonlinear Anal. Real World Appl.}, 54:103106, 24, 2020.

\bibitem{PredPreyHunting}
Micka\"{e}l Teixeira~Alves and Frank~M. Hilker.
\newblock Hunting cooperation and {A}llee effects in predators.
\newblock {\em J. Theoret. Biol.}, 419:13--22, 2017.

\end{thebibliography}
\end{document}